\documentclass[dvips,11pt,preprint]{amsart}
\RequirePackage[OT1]{fontenc}
\RequirePackage{amsthm,amsmath}
\usepackage{amsthm,amsmath}
\theoremstyle{plain}
\newtheorem{theorem}{Theorem}[section]
\newtheorem{lemma}{Lemma}[section]
\newtheorem{proposition}{Proposition}[section]
\newtheorem{definition}{Definition}[section]
\newtheorem{corollary}{Corollary}[section]
\makeatother
\begin{document}

% "Title of the paper"
\title{Asymptotics for regression models under loss of identifiability}
%\runtitle{Asymptotics for regression models}

% indicate corresponding author with \corref{}
% \author{\fnms{John} \snm{Smith}\corref{}\ead[label=e1]{smith@foo.com}\thanksref{t1}}
% \thankstext{t1}{Thanks to somebody} 
% \address{line 1\\ line 2\\ printead{e1}}
% \affiliation{Some University}

\author{\textsc J. Rynkiewicz \\ \textit { SAMM, Universite Paris 1} }

\address{SAMM, Universit\'e de Paris 1}
%\and
%\author{\fnms{???} \snm{???}\ead[label=e2]{???}}
%\address{\printead{e2}}
%\affiliation{Universit\'e de Paris 1}

%\runauthor{J. Rynkiewicz}
\maketitle
\begin{abstract}
This paper discusses the asymptotic behavior of regression models under general conditions. First, we give a general inequality for the difference of the sum of square errors (SSE) of the estimated regression model and the SSE of the theoretical best regression function in our model. A set of  generalized derivative functions is a key tool in deriving such inequality. Under suitable Donsker condition for this set, we give the asymptotic distribution for the difference of SSE. We show how to get this Donsker property for parametric models even if the parameters characterizing the best regression function are not unique. This result is applied to neural networks regression models with redundant hidden units when loss of identifiability occurs. 
\end{abstract}

%\begin{keyword}[class=AMS]
%\kwd[Primary ]{62H10}
%\kwd{}
%\kwd[; secondary ]{62F12}
%\end{keyword}

%\begin{keyword}
%\kwd{regression models}
%\kwd{Donsker class}
%\kwd{loss of identifiability}
%\kwd{multilayer neural networks}
%\end{keyword}

\section{introduction}
This paper discusses the asymptotic behavior of regression models under general conditions. Let $\mathcal F$ be the family of possible regression functions and suppose that we observe a random sample 
\[
(X_1Y_1),\cdots,(X_n,Y_n),
\] from the distribution $P$ of a vector $(X,Y)$, with $Y$ a real random variable, that follows the regression model
\begin{equation}
\label{Regression}
Y=f_0(X)+\varepsilon,\mbox{ }E\left(\varepsilon\left|X\right.\right)=0,\mbox{ }E\left(\varepsilon^2\left|X\right.\right)=\sigma^2<\infty.
\end{equation}
In our model, the function $f_0$ will be the best regression function among the set $\mathcal F$:
\[
f_0=\arg\min_{f\in\mathcal F}\Vert Y-f(X)\Vert_2,
\]
where \[
\Vert g(Z)\Vert_2:=\sqrt{\int g(z)^2dP(z)}
\]
is the ${\mathcal L}^2$  norm for an square integrable function $g$.

For simplicity, we assume that the best function $f_0$ is unique.

A natural estimator of $f_0$ is the least square estimator (LSE) $\hat f$ that minimizes the sum of square errors (SSE): 
\begin{equation}
\label{LSE}
\hat f=\arg\min_{f\in\mathcal F} \sum_{t=1}^n(Y_t-f(X_t))^2.
\end{equation}
$\hat f$ should be expected to converge to the function $f_0$ under suitable conditions.
If ${\mathcal F}$ is a parametric family and $\Theta$ is a set of possible parameters, ${\mathcal F}=\left\{f_\theta,\theta\in\Theta\right\}$, the LSE is the parameter $\hat \theta$ that minimizes
\begin{equation}
\label{ParamLSE}
\hat \theta= \arg\min_{\theta\in\Theta}\sum_{t=1}^n(Y_t-f_\theta(X_t))^2.
\end{equation}
Let us write $\Theta_0$ the set of parameters realizing the best regression function $f_0$: $\forall \theta\in\Theta_0$, $f_\theta=f_0$. 
If the set  ${\mathcal F}$ is large enough, it may be possible that the dimension of the interior of the set $\Theta_0$ is larger than zero and various difficulties arise in analyzing the statistical properties of estimators of $f_0$. This is for example the case if $\mathcal F$ contains multilayer neural networks with redundant hidden units (see \cite{Fukumizu2003}).  

Under loss of identifiability of the parameters, the asymptotics for likelihood functions has been studied by \cite{Liu} who improve the method of \cite{Dacunha} and\\
\cite{Dacunha2}. The authors establish a general quadratic approximation of the log-likelihood ratio in a neighborhood of the true density which is valid with or without loss of identifiability of the parameter of the true distribution. In this paper, we will use a similar idea, but here we are interested in regression functions, not in density functions, so we will introduce generalized derivative functions: 
\begin{equation}
\label{Generalizedderivative}
d_f(x)=\frac{f(x)-f_0(x)}{\Vert f(X)-f_0(X)\Vert_2},\, f\neq f_0.
\end{equation}

Under some general regularity conditions, this paper shows that 
\begin{equation}\label{CVSS}
\lim_{n\rightarrow\infty}\sum_{t=1}^n\left(Y_t-f_0(X_t)\right)^2-\left(Y_t-\hat f(X_t)\right)^2=\sigma^2\sup_{s\in\mathcal D}W_s^2
\end{equation} 
where $\mathcal D$ is the ${\mathcal L}^2$ limits of the generalized derivative function $d_f$ as $\Vert f(X)-f_0(X)\Vert_2\rightarrow 0$. 
Such result allows for example, to fully explicit the asymptotic behavior of the SSE when regression functions are multilayer neural networks, even if $\mathcal F$ is too big and contains neural networks with redundant hidden units. 
 
This result is a consequence of the very general inequality: For all regression function $f\in{\mathcal F}$, $f\neq f_0$,
\begin{equation}\label{ineg}
\sum_{t=1}^n\left(Y_t-f_0(X_t)\right)^2-\left(Y_t-f(X_t)\right)^2\leq\frac{\left(\frac{\sum_{t=1}^{n}\varepsilon_td_f(X_t)}{\sqrt{n}}\right)^2}{\frac{\sum_{t=1}^{n}\left(d_f(X_t)\right)^2}{n}}
\end{equation}

and the fact that the empirical process :
\begin{equation}\label{empiricalProcess}
\frac{1}{\sqrt n}\sum_{t=1}^n\varepsilon_td_f(X_t)
\end{equation}
converges in distribution to some Gaussian process. For instance, when ${\mathcal S}=\left\{d_f,f\in{\mathcal F},f\neq f_0\right\}$ is a Donsker class, $\frac{1}{\sqrt n}\sum_{t=1}^n\varepsilon_td_f(X_t)$ converges uniformly to some zero-mean Gaussian process.

Note that, even if the set $\mathcal F$ is a regular parametric family, the function $\theta\mapsto d_{f_\theta}(x)$ may be not extendable by continuity in $\theta_0\in\Theta_0$, hence the Donsker property of the set of generalized derivative functions has to be carefully studied. This problem occurs also for the generalized score functions $S_\theta$ of \cite{Liu}, although the authors did not mention it. 

The paper is organized as follows: Section 2 establishes  the asymptotic distribution of the SSE for regression models if the set of generalized derivative functions ${\mathcal S}$ is Donsker. In the next section, we show how to get the Donsker property for ${\mathcal S}$ in the parametric case but under loss of identifiability.  As an example, section 4 characterizes the asymptotic distribution of regression using neural networks with redundant hidden units. 

\section{Asymptotic distribution of the SSE}
For sake of simplicity we consider identically distributed independent variables, but all the following results can be easily generalized to geometrically mixing stationary sequence of random variables as in \\
\cite{Olteanu} or \cite{Gassiat}. For example, our results may be applied to non-linear autoregressive models using multilayer neural neural networks as in \cite{Yao}.  
Under fairly general condition the LSE is consistent and the regularity conditions of this  paper imply consistency, so the asymptotic distribution of SSE is determined by the local properties of the regression function in a small ${\mathcal L}^2$-neighborhood of the best regression function $f_0$. 

First we begin with some definitions.
\begin{definition}

The envelope function of a class of functions $\mathcal F$ is defined as 
\[F(x)\equiv \sup_{f\in \mathcal F}\left|f(x)\right|.
\]
 We will use the abbreviation $Pf=\int fdP$ for an integrable function $f$ and a probability measure $P$. A family of random sequences 
\[
\left\{Y_n(g), g\in{\mathcal G},n=1,2,\cdots\right\}
\]
is said to be uniformly $O_P(1)$ if for every $\delta>0$, there exist constants $M>0$ and $N(\delta,M)$ such that 
\[
P\left(\sup_{g\in\mathcal G}|Y_n(g)|\leq M\right)\geq 1-\delta
\]
for all $n\geq N(\delta,M)$.

A family of random sequences 
\[
\left\{Y_n(g), g\in{\mathcal G},n=1,2,\cdots\right\}
\]
is said to be uniformly $o_P(1)$ if for every $\delta>0$ and $\varepsilon>0$ there exists a constant $N(\delta,\varepsilon)$ such that 
\[
P\left(\sup_{g\in\mathcal G}|Y_n(g)|<\varepsilon\right)\geq 1-\delta
\]
for all $n\geq N(\delta,\varepsilon)$.
\end{definition}  
\subsection{Upper bound for the SSE} 
We prove this lemma which gives a very general upper bound for the sum of square errors.

\begin{lemma}\label{lemmeineg}
For all regression function $f\in{\mathcal F}$ with $f\neq f_0$:
\[
\sum_{t=1}^n\left(Y_t-f_0(X_t)\right)^2-\left(Y_t-f(X_t)\right)^2\leq\frac{\left(\frac{\sum_{t=1}^{n}\varepsilon_td_f(X_t)}{\sqrt{n}}\right)^2}{\frac{\sum_{t=1}^{n}\left(d_f(X_t)\right)^2}{n}}.
\]
\end{lemma}
\begin{proof}
We have
\[
\begin{array}{l}
\sum_{t=1}^n\left(Y_t-f_0(X_t)\right)^2-\left(Y_t-f(X_t)\right)^2=\\
\sum_{t=1}^n\left(Y_t-f_0(X_t)\right)^2-\left(Y_t-f_0(X_t)+f_0(X_t)-f(X_t)\right)^2=\\
\sum_{t=1}^n2\varepsilon_t\left(f(X_t)-f_0(X_t)\right)-\left(f(X_t)-f_0(X_t)\right)^2
\end{array}
\]
Now, let us write
\[
\begin{array}{l}
A=\Vert f_0(X_t)-f(X_t) \Vert_2\times\sqrt{\sum_{t=1}^{n}\left(\frac{f(X_t)-f_0(X_t)}{\Vert f(X_t)-f_0(X_t) \Vert_2}\right)^2}\\
\mbox{and}\\
Z=\frac{\sum_{t=1}^{n}\varepsilon_t\frac{f(X_t)-f_0(X_t)}{\Vert f(X_t)-f_0(X_t) \Vert_2}}{\sqrt{\sum_{t=1}^{n}\left(\frac{f(X_t)-f_0(X_t)}{\Vert f(X_t)-f_0(X_t) \Vert_2}\right)^2}},
\end{array}
\]
then remark that $2AZ-A^2\leq Z^2$ implies that
\[
\sum_{t=1}^n\left(Y_t-f_0(X_t)\right)^2-\left(Y_t-f(X_t)\right)^2\leq\frac{\left(\frac{\sum_{t=1}^{n}\varepsilon_t\frac{f(X_t)-f_0(X_t)}{\Vert f(X_t)-f_0(X_t) \Vert_2}}{\sqrt{n}}\right)^2}{\frac{\sum_{t=1}^{n}\left(\frac{f(X_t)-f_0(X_t)}{\Vert f(X_t)-f_0(X_t) \Vert_2}\right)^2}{n}}.
\]
\end{proof}
\subsection{Approximation of the SSE}
Define the limit-set of derivatives ${\mathcal D}$ as the set of functions $d\in L^2(P)$ such that one can find a sequence $(f_n)\in\mathcal F$ satisfying 
\(
\Vert f_n(X)-f_0(X)\Vert_2 \xrightarrow[n\rightarrow\infty]{}0
\) 
and $\Vert d-d_{f_n}\Vert_2 \xrightarrow[n\rightarrow\infty]{}0$.
With such $(f_n)$, define, for all $t\in[0,1]$, $f_t=f_n$, where $n\leq \frac{1}{t}<n+1$. We thus have that, for any $d\in{\mathcal D}$, there exists a parametric path $(f_t)_{0\leq t\leq\alpha}$ such that for any $t\in[0,\alpha]$, $f_t\in{\mathcal F}$, $t\mapsto\Vert f_t(X)-f_0(X)\Vert_2 $ is continuous, tends to $0$ as $t$ tends to 0 and  $\Vert d-d_{f_t}\Vert_2 \rightarrow 0$ as $t$ tends to 0. Using the reparameterization $\Vert f_u(X)-f_0(X)\Vert_2=u$, for any $d\in{\mathcal D}$, there exists a parametric path $(f_u)_{0\leq u\leq \alpha}$ such that:
\begin{equation}\label{Parampath}
\int \left(f_u-f-ud\right)^2dP=o(u^2).
\end{equation} 

Now, let us state the following theorem:
 
\begin{theorem}\label{asymptotM}
If the set of generalized derivative function ${\mathcal S}$ is a Donsker class and for any $d$ in the limit-set of derivatives ${\mathcal D}$, a reparameterization $(f_u)_{0\leq u\leq \alpha}$  exists so that $\Vert d-d_{f_u}\Vert_2 \rightarrow 0$ as $u$ tends to 0 and the map 
\[
u\mapsto P(Y-f_u(X))^2
\] 
admits a second-order Taylor expansion  with strictly positive second derivative $\frac{\partial^2P(Y-f_u(X))^2}{\partial u^2}$ at $u=0$, then
\[
\sup_{f\in{\mathcal F}}\sum_{t=1}^n\left(Y_t-f_0(X_t)\right)^2-\left(Y_t-f(X_t)\right)^2=\sup_{d\in\mathcal D}\left(\frac{1}{\sqrt{n}}\sum_{t=1}^n\varepsilon_td(X_t)\right)^2+o_P(1).
\]
\end{theorem}

\begin{proof}
We have
\[
\begin{array}{l}
\sum_{t=1}^n\left(Y_t-f_0(X_t)\right)^2-\left(Y_t-f(X_t)\right)^2=2\Vert f(X)-f_0(X)\Vert_2\sum_{t=1}^n\varepsilon_tdf(X_t)\\
-\Vert f(X)-f_0(X)\Vert_2^2\sum_{t=1}^ndf^2(X_t).
\end{array}
\]
As soon as $\sum_{t=1}^n\left(Y_t-f_0(X_t)\right)^2-\left(Y_t-f(X_t)\right)^2\geq 0$, 
\[
\begin{array}{l}
2\Vert f(X)-f_0(X)\Vert_2\sum_{t=1}^n\varepsilon_tdf(X_t)\geq \Vert f(X)-f_0(X)\Vert_2^2\sum_{t=1}^ndf^2(X_t)
\end{array}
\]
and
\begin{equation}\label{ineg2}
\sup_{f\in{\mathcal F},\sum_{t=1}^n\left(Y_t-f_0(X_t)\right)^2-\left(Y_t-f(X_t)\right)^2\geq 0}\Vert f(X)-f_0(X)\Vert_2\leq2\sup_{f\in{\mathcal F}}\frac{\sum_{t=1}^n \varepsilon_tdf(X_t)}{\sum_{t=1}^ndf^2(X_t)}.
\end{equation}
Since, $\mathcal S$ is Donsker
\begin{equation}\label{compar}
\sup_{f\in{\mathcal F}}\frac{1}{n}\left(\sum_{t=1}^n\varepsilon_td_f(X_t)\right)^2=O_P(1)
\end{equation}
and $S$ admits an envelope function $F$ such that $P(F^2)< \infty$, so $S^2$ is Glivenko-Cantelli and
\begin{equation}\label{compar2}
\sup_{f\in{\mathcal F}}\left|\frac{1}{n}\sum_{t=1}^nd_f^2(X_t)-1\right|=o_P(1).
\end{equation}
Then, one may apply inequality (\ref{ineg2}) to obtain
\begin{equation}\label{compar3}
\sup_{f\in{\mathcal F},\sum_{t=1}^n\left(Y_t-f_0(X_t)\right)^2-\left(Y_t-f(X_t)\right)^2\geq 0}\Vert f(X)-f_0(X)\Vert_2=O_P\left(\frac{1}{\sqrt{n}}\right).
\end{equation}
By lemma \ref{lemmeineg}, 
\[
\sup_{f\in{\mathcal F}}\sum_{t=1}^n\left(Y_t-f_0(X_t)\right)^2-\left(Y_t-f(X_t)\right)^2\leq\sup_{f\in{\mathcal F}}\frac{\left(\frac{\sum_{t=1}^{n}\varepsilon_t\frac{f_0(X_t)-f(X_t)}{\Vert f_0(X_t)-f(X_t) \Vert_2}}{\sqrt{n}}\right)^2}{\frac{\sum_{t=1}^{n}\left(\frac{f_0(X_t)-f(X_t)}{\Vert f_0(X_t)-f(X_t) \Vert_2}\right)^2}{n}}.
\]
Using (\ref{compar2}), we obtain that
\[
\begin{array}{l}
\sup_{f\in{\mathcal F}}\sum_{t=1}^n\left(Y_t-f_0(X_t)\right)^2-\left(Y_t-f(X_t)\right)^2\\
\leq\sup_{f\in{\mathcal F}}\left(\frac{\sum_{t=1}^{n}\varepsilon_t\frac{f_0(X_t)-f(X_t)}{\Vert f_0(X_t)-f(X_t) \Vert_2}}{\sqrt{n}}\right)^2+o_P(1).
\end{array}
\]

Let \({\mathcal F}_n=\left\{f\in{\mathcal F} : \Vert f_n(X)-f_0(X)\Vert_2\leq n^{-1/4}\right\}\). Using (\ref{compar3}), we obtain that
\[
\begin{array}{l}
\sup_{f\in{\mathcal F}}\sum_{t=1}^n\left(Y_t-f_0(X_t)\right)^2-\left(Y_t-f(X_t)\right)^2\\
\leq\sup_{f\in{\mathcal F}_n}\left(\frac{\sum_{t=1}^{n}\varepsilon_t\frac{f_0(X_t)-f(X_t)}{\Vert f_0(X_t)-f(X_t) \Vert_2}}{\sqrt{n}}\right)^2+o_P(1).
\end{array}
\]

Now, $\sup_{f\in{\mathcal F}_n}\Vert d_f-\mathcal D\Vert_2\xrightarrow[n\rightarrow\infty]{}0$, thus for a sequence $u_n$ decreasing to $0$, and with 
\[\Delta_n=\left\{d_f-d\ :\ f\in{\mathcal F}_n,\ d\in {\mathcal D},\Vert d_f-d\Vert_2\leq u_n\right\},\]
we obtain that
\[
\begin{array}{l}
\sup_{f\in{\mathcal F}}\sum_{t=1}^n\left(Y_t-f_0(X_t)\right)^2-\left(Y_t-f(X_t)\right)^2\\
\leq\left(\sup_{d\in\mathcal D}\frac{\sum_{t=1}^{n}\varepsilon_td\left(X_t\right)}{\sqrt{n}}+\sup_{\delta\in\Delta_n}\frac{\sum_{t=1}^{n}\varepsilon_t\delta\left(X_t\right)}{\sqrt{n}}\right)^2+o_P(1).
\end{array}
\]
But, using the Donsker property, the definition of $\Delta_n$ and the property of asymptotic stochastic equicontinuity of empirical processes indexed by a Donsker class, we get:
\[
\sup_{\delta\in\Delta_n}\frac{\sum_{t=1}^n\varepsilon_t\delta\left(X_t\right)}{\sqrt{n}}=o_{P}(1),
\] 
and
\begin{equation}\label{inegscore}
\begin{array}{l}
\sup_{f\in{\mathcal F}}\sum_{t=1}^n\left(Y_t-f_0(X_t)\right)^2-\left(Y_t-f(X_t)\right)^2\\
\leq\sup_{d\in\mathcal D}\left(\frac{\sum_{t=1}^{n}\varepsilon_td\left(X_t\right)}{\sqrt{n}}\right)^2+o_P(1).
\end{array}
\end{equation}

Moreover, since $S$ admits a square integrable envelope function, a function $m$ exists such that for $u_1$ and $u_2$ belonging to a parametric path converging to a limit function $d$:
\[
\left|(y-f_{u_1}(x))^2-(y-f_{u_2}(x))^2\right|\leq m(x,y)|u_1-u_2|
\]
and since, along a path, the map 
\[
u\mapsto P(Y-f_u(X))^2
\] 
admits a second-order Taylor expansion  with strictly positive second derivative $\frac{\partial^2P(Y-f_u(X))^2}{\partial u^2}$ at $u=0$, we can use classical normal asymptotic theorem for M-estimators (see theorem 5.23 of \cite{Vandervaart}) along this parametric paths, to obtain a sequence of finite subsets ${\mathcal D}_k$ increasing to $\mathcal D$ such that
\[
\begin{array}{l}
\sup_{f\in{\mathcal F}}\sum_{t=1}^n\left(Y_t-f_0(X_t)\right)^2-\left(Y_t-f(X_t)\right)^2\\
\geq\sup_{d\in{\mathcal D}_k}\left(\frac{\sum_{t=1}^{n}\varepsilon_td\left(X_t\right)}{\sqrt{n}}\right)^2+o_P(1).
\end{array}
\]
for any $k$, therefore, equality holds in (\ref{inegscore}).
\end{proof}

Define $\left(W(d)\right)_{d\in{\mathcal D}}$ the centered Gaussian process with covariance the scalar product in $L^2(P)$, an immediate application of theorem \ref{asymptotM} gives:
\begin{corollary}\label{corollasymp}
Assume that ${\mathcal S}$ is a Donsker class, 
\[
\sup_{f\in{\mathcal F}}\sum_{t=1}^n\left(Y_t-f_0(X_t)\right)^2-\left(Y_t-f(X_t)\right)^2
\] 
converges in distribution to
\[
\sigma^2\sup_{d\in\mathcal D}\left(W(d)\right)^2.
\]
\end{corollary}
As we see, the Donsker property of the set of generalized derivatives functions ${\mathcal S}$ is fundamental for the results above. In the next section we will show how to get it for parametric models under loss of identifiability. 

\section{Donsker property for ${\mathcal S}$}
This section will give a framework for the demonstration of Donsker property for the set of generalized derivative functions $\mathcal S$ for parametric models and under loss of identifiability. Note that this framework could be easily adapted to likelihood ratio test and generalized score functions of \cite{Liu}. 

First, we recall the notion of bracketing entropy. Consider the set $\mathcal{S}$  endowed with the norm $\left\Vert \cdot\right\Vert_{2}$. For every $\eta>0$, we define an $\eta$-bracket by $\left[l,\, u\right]=\left\{ f\in\mathcal{S},\, l\leq f\leq u\right\} $ such that $\left\Vert u-l\right\Vert_{2} <\eta$. The $\eta$-bracketing entropy is
\[\mathcal{H}_{\left[\cdot\right]}\left(\eta,\mathcal{S},\left\Vert \cdot\right\Vert_{2}\right)=\ln \left(\mathcal{N}_{\left[\cdot\right]}\left(\eta,\mathcal{S},\left\Vert \cdot\right\Vert_{2} \right)\right),\]
where $\mathcal{N}_{\left[\cdot\right]}\left(\eta,\mathcal{S},\left\Vert \cdot\right\Vert_{2} \right)$
is the minimum number of $\eta$-brackets necessary to cover $\mathcal{S}$.

With the previous notations if
\[ \int_{0}^{1}\sqrt{\mathcal{H}_{\left[\cdot\right]}\left(\eta,\mathcal{S},\left\Vert \cdot\right\Vert_{2} \right)}d\eta<\infty,\]
then, according to \cite{Vandervaart}, the set ${\mathcal S}$ is Donsker. Note that, if the  number of $\eta$-brackets necessary to cover $\mathcal{S}$, $\mathcal{N}_{\left[\cdot\right]}\left(\eta,\mathcal{S},\left\Vert \cdot\right\Vert_{2} \right)$, is a polynomial function of $\frac{1}{\eta}$, $\mathcal{S}$ will be Donsker. If a class of function 
\[
{\mathcal F}=\left\{f_\theta,\theta\in\Theta\subset{\mathcal R}^D\right\}
\] 
is parametric and regular,  in general, for any $\theta_1$, $\theta_2\in\Theta$ there exists a function $G\in L^2(P)$ such that
\begin{equation}\label{modulusnum}
\left|f_{\theta_1}(x)-f_{\theta_2}(x)\right|\leq \Vert \theta_1-\theta_2\Vert G(x)
\end{equation}
and according to \cite{Vandervaart} a constant $K$ exists such that, 
\[
\mathcal{N}_{\left[\cdot\right]}\left(\eta,\mathcal{S},\left\Vert \cdot\right\Vert_{2} \right)\leq K\left(\frac{\mbox{diam}\Theta}{\eta}\right)^D.
\]
Hence, the set ${\mathcal F}$ is Donsker. However, even if the set ${\mathcal F}$ is parametric and regular, the set ${\mathcal S}=\left\{d_{f_\theta}=\frac{f_\theta-f_0}{\Vert f_\theta-f_0\Vert_2},\theta\in\Theta,f_\theta\neq f_0\right\}$ is not regular, since $\theta\mapsto d_{f_\theta}(x)$ is, in general, not extendable by continuity in $\theta_0$ a parameter realizing the best regression function $f_0$. Note that, it is also the case for the generalized score function $S_\theta$ of \cite{Liu}, in particular in the case of finite mixture models under loss of identifiability. Hopefully, we can show the Donsker property of the set $\mathcal S$ by an other method which can be applied also to generalized score function in the framework of likelihood ratio test as in \cite{Olteanu}.

For proving that $\mathcal{N}_{\left[\cdot\right]}\left(\eta,\mathcal{S},\left\Vert \cdot\right\Vert_{2} \right)$, is a polynomial function of $\frac{1}{\eta}$, we have to  split $\mathcal{S}$ into two sets of functions: A set in a neighborhood of the best regression function $f_0$ and a second one at a distance at least $\eta$ of $f_0$. For a sufficiently small $\eta>0$, we consider $\mathcal{F}_{\eta}\subset\mathcal{F}$, a ${\mathcal L}^2$-neighborhood of $f_0$: 
$\mathcal{F}_{\eta}=\left\{ f_\theta\in {\mathcal F},\:\left\Vert f_\theta-f_0\right\Vert_2\leq\eta,\: f_\theta\neq f_0\right\} $. 
$\mathcal{S}$ is split into $\mathcal{S}_{\eta}=\left\{ d_{f_\theta},\: f_\theta\in\mathcal{F}_{\eta}\right\} $ and $\mathcal{S}\setminus\mathcal{S}_{\eta}$.

On $\mathcal{S}\setminus\mathcal{S}_{\eta}$, it can be easily seen that 

\[
\left\Vert d_{f_{\theta_1}}-d_{f_{\theta_2}}\right\Vert_2\leq\frac{\left\Vert f_{\theta_1}-f_{\theta_2}\right\Vert_2}{\left\Vert f_{\theta_1}-f_0\right\Vert_2}+\left\Vert \frac{f_{\theta_2}-f_0}{\left\Vert f_{\theta_1}-f_0\right\Vert_2}-\frac{f_{\theta_2}-f_0}{\left\Vert f_{\theta_2}-f_0\right\Vert_2}\right\Vert_2\]

for every $f_{\theta_1},f_{\theta_2}\in\mathcal{F}\setminus\mathcal{F}_{\eta}$.
By (\ref{modulusnum}),  if $\Vert \theta_1-\theta_2\Vert\leq \eta^3$, a constant $C$ exists such that 
\[
\left\Vert f_{\theta_1}-f_{\theta_2}\right\Vert_2\leq C\eta^3.
\]
Then,  by the definition of $\mathcal{S}_{\eta}$,
\[
\begin{array}{l}
\left\Vert \frac{f_{\theta_2}-f_0}{\left\Vert f_{\theta_1}-f_0\right\Vert_2}-\frac{f_{\theta_2}-f_0}{\left\Vert f_{\theta_2}-f_0\right\Vert_2}\right\Vert_2\\
\leq\frac{\left\Vert f_{\theta_2}-f_0\right\Vert_2}{\eta}\left(\frac{1}{1+\eta^2}-1\right)=\left\Vert f_{\theta_2}-f_0\right\Vert_2\left(\eta+o(\eta)\right)
\end{array}
\]
and, a constant $M$ exists so that

\[
\left\Vert d_{f_{\theta_1}}-d_{f_{\theta_2}}\right\Vert_2\leq C\eta^2+\left\Vert f_{\theta_2}-f_0\right\Vert_2\left(\eta+o(\eta)\right)\leq M\eta.
\]

Finally, we get:

\[
\mathcal{N}_{\left[\cdot\right]}\left(\eta,\mathcal{S}\setminus\mathcal{S}_{\eta},\left\Vert \cdot\right\Vert _{2}\right)=\mathcal{O}\left(\frac{1}{\eta^3}\right)^{D}=\mathcal{O}\left(\frac{1}{\eta}\right)^{3D}\]
where $D$ is the dimension of parameter vector of the model.

It remains to prove that the bracketing number is a polynomial function of ($\frac{1}{\eta}$) for $\mathcal{S}_{\eta}$. The idea is to reparameterize the model in a convenient manner which will allow a Taylor expansion around the identifiable part of the true value of the parameters, then, using this Taylor expansion, we can show that the bracketing number of $\mathcal{S}_{\eta}$ is a polynomial function of $\frac{1}{\eta}$. Indeed, in many applications,  there exists a reparameterization $\theta\mapsto\left(\phi,\psi\right)$ such that $f_\theta=f_0$ is equivalent to the condition that $\phi=\phi_0$ for all $\psi$. Then, positive integers $(q_0,q_1)$ and linearly independent functions $g_{\beta_{i}^0}$, $g_{\beta^0_i}^{\prime}$, $g_{\beta^0_i}^{\prime\prime}$, $i=1,...,q_{0}$, $g_{\beta_j}, j=1,\cdots,q_1$ exist so that the difference of regression functions can be written:
\begin{equation} \label{lrts}
\begin{array}{l}
f_\theta-f_0=f_{(\phi,\psi)}-f_0=\left(\phi-\phi_0\right)^T\frac{\partial f_{(\phi_0,\psi)}}{\partial\phi}\\
+\frac{1}{2}\left(\phi-\phi_0\right)^T\frac{\partial^2f_{(\phi_0,\psi)}}{\partial\phi^2}\left(\phi-\phi_0\right)+o(\Vert f_{(\phi,\psi)}-f_0\Vert_2^2)=\\
\sum_{i=1}^{q_{0}}\alpha_{i}g_{\beta_{i}^{0}}+\sum_{i=1}^{q_1}\nu_ig_{\beta_i}+\sum_{i=1}^{q_{0}}\delta_{i}^{T}g^{\prime}_{\beta^0_i}+\sum_{i=1}^{q_{0}}\gamma_{i}^{T}g^{\prime\prime}_{\beta^0_i}\gamma_{i}+o\left(\Vert f_{(\phi,\psi)}-f_0\Vert_2^2\right)
\end{array}
\end{equation}
 
where $\left(\beta_{i}^0\right)_{1\leq i\leq q_0}$ are fixed parameter, $\alpha_i\mbox{ and }\nu_i$ are real parameters and $\delta_i\mbox{ and }\gamma_i$ are parameter vectors with sizes compatible with functions $g_{\beta^0_i}^{\prime}$ and $g_{\beta^0_i}^{\prime\prime}$, $\left(\beta_{i}\right)_{1\leq i\leq q_1}$ are in a compact set and inequality (\ref{modulusnum}) is true for the regular parametric functions $\left(g_{\beta_i}\right)_{1\leq i\leq q_1}$. Moreover $\Vert f_{(\phi,\psi)}-f_0\Vert_2^2\leq \eta^2$ on $\mathcal{S}_{\eta}$. 

We will see an example of such expansion in the next section. Note that similar expansion is also possible for the likelihood ratio framework (see \cite{Liu}, section 4). We get then the next result:

\begin{proposition}\label{polycov}
If an expansion like (\ref{lrts}) exists, a positive integer $d$ exists so that the number of $\eta$-brackets $\mathcal{N}_{\left[\cdot\right]}\left(\eta,\mathcal{S}_{\eta},\left\Vert \cdot\right\Vert _{2}\right)$
covering $\mathcal{S}_{\eta}$ is $\mathcal{O}\left(\frac{1}{\eta}\right)^{d}$.
\end{proposition}
\begin{proof}
The idea is to bound $\mathcal{N}_{\left[\cdot\right]}\left(\eta,\mathcal{S}_{\eta},\left\Vert \cdot\right\Vert _{2}\right)$
by the number of $\eta$-brackets covering a wider class of
functions. For every $f_\theta\in\mathcal{F}_{\eta}$, we will consider the
reparameterization $\theta\mapsto \left(\phi_{t},\psi_{t}\right)$ which allows to get a second-order development of the density ratio like (\ref{lrts}).

Now, using the linear independence of functions $g_{\beta_{i}}$, $g_{\beta_{i}^0}$, $g_{\beta^0_i}^{\prime}$, $g_{\beta^0_i}^{\prime\prime}$,  for every vector
\(v=\left(\alpha_{i},\delta_{i},\gamma_{i},i=1,\cdots,q_0,\nu_{j},j=1,\cdots,q_1\right)\)
of norm $1$,
\[
\left(v,(\beta_i)_{1\leq i\leq q_1}\right)\mapsto\left\Vert \sum_{i=1}^{q_{0}}\alpha_{i}g_{\beta_{i}^{0}}+\sum_{i=1}^{q_1}\nu_ig_{\beta_i}+\sum_{i=1}^{q_{0}}\delta_{i}^{T}g^{\prime}_{\beta_i}+\sum_{i=1}^{q_{0}}\gamma_{i}^{T}g^{\prime\prime}_{\beta_i}\gamma_{i}\right\Vert_2> 0.\]
Using the compacity of sets 
\[
\begin{array}{l}
{\mathcal V}=\left\{v=\left(\alpha_{i},\delta_{i},\gamma_{i},i=1,\cdots,q_0,\nu_{j},j=1,\cdots,q_1\right),\Vert v\Vert=1\right\}\\
\mbox{and}\\
\left\{(\beta_i)_{1\leq i\leq q_1}\right\}
\end{array}
\] 
$m>0$ exists so that for all $(\beta_i)_{1\leq i\leq q_1}$ and $v\in{\mathcal V}$,
\[
\left\Vert \sum_{i=1}^{q_{0}}\alpha_{i}g_{\beta_{i}^{0}}+\sum_{i=1}^{q_1}\nu_ig_{\beta_i}+\sum_{i=1}^{q_{0}}\delta_{i}^{T}g^{\prime}_{\beta_i}+\sum_{i=1}^{q_{0}}\gamma_{i}^{T}g^{\prime\prime}_{\beta_i}\gamma_{i}\right\Vert_2\geq m.\]
At the same time, since

\[
\left\Vert \frac{f_{\left(\phi_{t},\psi_{t}\right)}-f_0}{\left\Vert f_{\left(\phi_{t},\psi_{t}\right)}-f_0\right\Vert_2}\right\Vert_2=1,
\]

the Euclidean norm of coefficients \(\left(\alpha_{i},\delta_{i},\gamma_{i},i=1,\cdots,q_0,\nu_{i},i=1,\cdots,q_1\right)\) in the development of $\frac{f_{\left(\phi_{t},\psi_{t}\right)}-f_0}{\left\Vert f_{\left(\phi_{t},\psi_{t}\right)}-f_0\right\Vert_2}$
is upper bounded by $\frac{1}{m}+1$. This fact implies that $\mathcal{S}_{\eta}$
can be included in

\[
\mathcal{H}=
\begin{array}{l}
\left\{ \sum_{i=1}^{q_{0}}\left(\alpha_{i}g_{\beta_{i}^{0}}+\delta_{i}^{T}g_{\beta^0_i}^{\prime}+\gamma_{i}^{T}g_{\beta^0_i}^{\prime\prime}\gamma_{i}\right)+\sum_{i=1}^{q_1}\nu_ig_{\beta_i}+C,\right.\\
\left.\left\Vert \left(\alpha_{i},\delta_i,\gamma_i,i=1,\cdots,q_0,\nu_i,i=1,\cdots,q_1\right)\right\Vert \leq\frac{1}{m}+1,\left|C\right|\leq\frac{1}{m}+1\right\} 
\end{array}
\]

and a positive integer $d$ exists so that $\mathcal{N}_{\left[\cdot\right]}\left(\eta,\mathcal{H},\left\Vert \cdot\right\Vert _{2}\right)=\mathcal{O}\left(\frac{1}{\eta}\right)^{d}$. 
\end{proof}
Note that, since the $\mathcal{N}_{\left[\cdot\right]}\left(\eta,\mathcal{S},\left\Vert \cdot\right\Vert_{2} \right)$, is a polynomial function of $\frac{1}{\eta}$, the Donsker property of $\mathcal S$ may be easily extended to $\beta$-mixing observations with respect to the norm $\Vert . \Vert_{2,\beta}$ (see \cite{Doukhan}).

\section{Application to regression with neural networks}
Feedforward neural networks or multilayer perceptrons (MLP) are well known and popular tools to deal with non-linear regression models.\\
\cite{White} reviews the statistical properties of  MLP estimation in detail, however he leaves an important question pending: The asymptotic behavior of the estimator when the MLP in use has redundant hidden units. When the noise of the regression model is assumed Gaussian,\\
\cite{Amari} give several examples of the behavior of the likelihood ratio test statistic (LRTS) in such cases. \cite{Fukumizu2003} shows that, for unbounded parameters, the  LRTS can have an order lower bounded by $O(\log(n))$ with $n$ the number of observations instead of the classical convergence property to a $\chi^2$ law. \cite{Hagiwara} investigate relation between LRTS divergence and weight size in a simple neural networks regression problem. 

However, if the set of possible parameters of the MLP regression model are bounded the behavior of LRTS and more generally the SSE is still unknown. 
In this section, we derive the distribution of the SSE if the parameters are in a compact (bounded and closed) set.

\subsection{The model}
\label{sec:1}
Let  $x=(x_1,\cdots,x_d)^T\in{\mathcal R}^{d}$ be the vector of inputs and \\
$w_i:=\left(w_{i1},\cdots,w_{id}\right)^T\in{\mathcal R}^{d}$ be the parameter vector of the hidden unit $i$. The MLP real function with $k$ hidden units can be written : 

\[
f_\theta(x)=\beta+\sum_{i=1}^k a_i\phi\left(w_i^Tx+b_i\right),
\]
with $\theta=\left(\beta,a_1,\cdots,a_k,b_1,\cdots,b_k,w_{11},\cdots,w_{1d},\cdots,w_{k1},\cdots,w_{kd}\right)$ the parameter vector of the model.  The transfer function $\phi$ will be assumed bounded and two times derivable. We assume also that the first and second derivatives of the transfer function $\phi$:  $\phi^{'}$ and $\phi^{''}$ are bounded like for sigmoid functions, the most used tranfer functions.  Moreover, in order to avoid a symmetry on the signs of the parameters, we assume that, for $1\leq i\leq k$, $a_i\geq 0$. Let $\Theta\subset {\mathcal R}\times{{\mathcal R}^+}^k\times{\mathcal R}^{k\times (d+1)}$ be the compact set of possible parameters, the regression model (\ref{Regression}) is then:
\[
Y=f_{\theta_0}(X)+\varepsilon
\]
with \(X\) is a random input variable with probability law $Q$ and 
\[
\theta_0=\left(\beta^0,a^0_1,\cdots,a^0_k,b^0_1,\cdots,b^0_k,w^0_{11},\cdots,w^0_{1d},\cdots,w^0_{k1},\cdots,w^0_{kd}\right)
\] 
a parameter such that $f_{\theta_0}=f_0$.  Note that the set of parameters $\Theta_0$ realizing the best regression function $f_0$ may belong to a non-null dimension sub-manifold if the number of hidden units is overestimated. Suppose, for example, we have a multilayer perceptron with two hidden units and the true function $f_{0}$ is given by a perceptron with only one hidden unit, say $f_{0}=a^0\tanh(w^0x)$, with $x\in\mathcal R$. Then, any parameter $\theta$ in the set:
\[
\left\{\theta\left|w_2=w_1=w^0,b_2=b_1=0,a_1+a_2=a^0\right.\right\}
\]
realizes the function $f_0$. Hence, classical statistical theory for studying the LSE can not be applied because it requires the identification of the parameters (up to some permutations and sign symmetries) so that the Hessian matrix of mean square error with respect to the parameters will be definite positive in a neighborhood of the parameter vector realizing the true regression function.
Let us denote $k_0$ the minimal number of hidden units to realize the best regression function $f_{0}$, we will compare the SSE of over-determined models against the true model : 
\[
\sum_{t=1}^{n}\left(Y_t-f_{0}(X_t)\right)^2-\sum_{t=1}^{n}\left(Y_t-f_{\theta}(X_t)\right)^2,
\]
when unidentifiability occurs (i.e. when $k>k_0$).
\subsection{Asymptotic distribution of the difference of SSE}
Let us give simple sufficient conditions for which the Donsker property of the generalized derivatives functions condition holds. Note that  assumption H-1 allows that, for any accumulation sequence of parameter $\theta_n$ leading to $f_0$, the regression functions $\left(f_{\theta_n}\right)$ are in a ${\mathcal L}^2$-neighborhood of $f_0$, in the same spirit of locally conic models of \cite{Dacunha}. Moreover, if the distribution $Q$ of the variable $X$ is not degenerated, it may be shown that the assumption H-3  is true for the sigmoid transfer function:
\[
\phi(t)=\tanh(t)
\] 
with straightforward extension of the results of \cite{Fukumizu1996}.  
\begin{description}
\item{H-1:}  $\Theta$ is a closed ball of  ${\mathcal R}\times{{\mathcal R}^+}^k\times{\mathcal R}^{k\times (d+1)}$ and its interior
contains parameters realizing the best regression function $f_{0}$. 
\item{H-2:} $E(\Vert X\Vert^4)<\infty$.
\item{H-3:} For distinct $\left(w_i,b_i\right)_{1\leq i\leq k},\mbox{ with } w_i \mbox{ not null}$, the functions of the set
\[
\begin{array}{l}
\left(1,\left(x_jx_l\phi^{''}({w^0_i}^Tx+b_i^0)\right)_{1\leq l \leq j\leq d,\ 1\leq i\leq k^0},\left(x_j\phi^{''}({w^0_i}^Tx+b_i^0)\right)_{1\leq j\leq d,\ 1\leq i\leq k^0}\right.\\
\left.\phi^{''}({w^0_i}^Tx+b_i^0)_{1\leq i\leq k^0}, \left(x_j\phi^{'}({w^0_i}^Tx+b_i^0)\right)_{1\leq j\leq d,\ 1\leq i\leq k^0}\right.\\
\left.\left(\phi^{'}({w^0_i}^Tx+b_i^0)\right)_{1\leq i\leq k^0},\left(\phi({w_i}^Tx+b_i)\right)_{1\leq i\leq k}\right)
\end{array}
\]
are linearly independent in the Hilbert space ${\mathcal L}^2(Q)$. 
\end{description}
We then get the following result:
 
\begin{theorem}\label{loilim}
Let the map $\Omega:{\mathcal L}^2(Q)\rightarrow {\mathcal L}^2(Q)$ be defined as 
\(
\Omega(f)=\frac{f}{\Vert f\Vert_2}.
\)
Under the assumptions H-1, H-2 and H-3, a centered Gaussian process $\{W(d),d\in{\mathcal D}\}$ with continuous sample paths and a covariance kernel $P\left(W(d_1)W(d_2)\right)=P\left(d_1d_2\right)$ exists so that 
\[
\lim_{n\rightarrow\infty}\sum_{t=1}^{n}\left(Y_t-f_{0}(X_t)\right)^2-\sum_{t=1}^{n}\left(Y_t-f_{\theta}(X_t)\right)^2=\sigma^2\sup_{d\in{\mathcal D}}\left(W(d)\right)^2.
\]
The index set ${\mathcal D}$ is defined as ${\mathcal D}=\cup_t{\mathcal D}_t$, the union runs over any possible $t=\left(t_1,\cdots,t_{k^0+1}\right)\in{\mathcal N}^{k^0+1}$ with  $0\leq t_1\leq k-k^0<t_2<\cdots <t_{k^0+1}\leq k$ and
\[\begin{array}{l}
{\mathcal D}_t=\left\{\Omega\left(\gamma+\sum_{i=0}^{k^0}\epsilon_i\phi({w^0_i}^TX+b_i^0)\right.\right.\\
\left.+\sum_{i=0}^{k^0}\phi^{'}({w^0_i}^TX+b_i^0)({\zeta}_{i}^TX+\alpha_i)\right.\\
+\delta(i)\sum_{i=1}^{k^0}\phi^{''}({w^0_i}^TX+b_i^0)\times\\
\left(\left(\sum_{j=t_{i}+1}^{t_{i+1}}{\nu_j}^TXX^T\nu_j+\eta_j{\nu_j}^TX+{\eta_j}^2\right)\right)\\
\left.+\sum_{i=t_{k^0+1}+1}^{k}\mu_i\phi({w_i}^TX+b_i)\right),\\
\gamma,\epsilon_1,\cdots,\epsilon_{k^0},\alpha_1,\cdots,\alpha_{k^0}, \eta_{t_1},\cdots,\eta_{t_{k^0+1}}\in{\mathcal R},\\
\mu_{t_{k^0+1}+1},\cdots,\mu_k\in {\mathcal R}^+;\zeta_1,\cdots,\zeta_{k^0},\nu_{t_1},\cdots,\nu_{t_{k^0+1}}\in{\mathcal R}^{d},\\
\left.(w_{{k^0+1}+1},b_{{k^0+1}+1}),\cdots,(w_{k},b_k)\in\Theta\backslash \left\{(w^0_1,b_1^0),\cdots,(w^0_{k^0},b_{k^0}^0)\right\}\right\}.
\end{array}
\]
$\delta(i)=1$ if a vector ${\bf q}$ exists  so that:\\
$q_j\geq 0$, $\sum_{j=t_{i}+1}^{t_{i+1}}q_j=1$, $\sum_{j=t_{i}+1}^{t_{i+1}}\sqrt{q_j}\nu_j=0$ and $\sum_{j=t_{i}+1}^{t_{i+1}}\sqrt{q_j}\eta_j=0$, otherwise $\delta(i)=0$.  
\end{theorem}

\begin{proof}

If the set of generalized derivative functions ${\mathcal S}=\left\{d_{f_\theta},\theta\in\Theta\backslash \Theta_0\right\}$ is a Donsker class, we can apply the theorem \ref{asymptotM} to conclude.
So, in order to show this Donsker  property,  we will get an asymptotic development of the generalized derivative function and apply proposition (\ref{polycov}).
\paragraph{Reparameterization.} 
The idea is similar of the reparameterization of finite mixture models in \cite{Liu}.  Under assumption H-3, the writing of $f_0$ with a neural network with $k_0$ hidden units is unique, up to some permutations:
\begin{equation}\label{BestMLP}
f_0=\beta^0+\sum_{i=1}^{k_0} a^0_i\phi\left({w^0_i}^Tx+b^0_i\right).
\end{equation}
Then, for a $\theta\in\Theta$, if $f_\theta=f_{0}$, a vector $t=(t_i)_{1\leq i\leq k^0+1}$ exists so that $0\leq t_1\leq k-k^0 <t_2<\cdots<t_{k^0+1}\leq k$ and, up to permutations, we have \(w_{1}=\cdots=w_{t_{1}}=0\) if $t_1>0$, \(\left(w_{t_{i}+1}=\cdots=w_{t_{i+1}}=w^0_i\right)_{1\leq i\leq k^0}\),  \(\left(b_{t_{i}+1}=\cdots=b_{t_{i+1}}=b^0_i\right)_{1\leq i\leq k^0}\),\\
\(\left(\sum_{j=t_{i}+1}^{t_{i+1}}a_j=a_i^0\right)_{1\leq i\leq k^0}\) and $\beta+\sum_{i=1}^{t_1}a_i\phi(b_i)=\beta^0$ if $t_1>0$ else $\beta=\beta_0$. 

For $1\leq i\leq k^0$, let us define $s_i=\sum_{j=t_{i}+1}^{t_{i+1}}a_j-a_i^0$ and,  if $\sum_{t_{i}+1}^{t_{i+1}}a_j\neq 0$, let us write $q_j=\frac{a_j}{\sum_{t_{i}+1}^{t_{i+1}}a_j}$. If $\sum_{t_{i}+1}^{t_{i+1}}a_j= 0$, $q_j$ will be set at $0$. Moreover, let us write $\gamma=\beta+\sum_{i=1}^{t_1}a_i\phi(b_i)-\beta^0$, if $t_1>0$ else $\gamma=\beta-\beta_0$. 

We get then the reparameterization $\theta\mapsto \left(\Phi_t,\psi_t\right)$ with 
\[
\begin{array}{l}
\Phi_t=\left(\gamma,(w_j)_{j={t_1}}^{t_{k^0+1}},(b_j)_{j={t_1}}^{t_{k^0+1}},(s_i)_{i=1}^{k^0},(a_j)_{t_{k^0+1}+1}^{k}\right),\\
\psi_t=\left((q_j)_{j={t_1}}^{t_{k^0+1}},(w_i,b_i)_{i=1+t_{k^0+1}}^{k}\right).
\end{array}
\]
With this parameterization, for a fixed $t$,  $\Phi_t$ is an identifiable parameter and all the non-identifiability of the model will be in $\psi_t$. Namely, $f_\theta$ will be equal to:
\[
\begin{array}{l}
f_\theta=(\gamma+\beta^0)+\sum_{i=1}^{k^0}(s_i+a^0_i)\sum_{j=t_{i-1}+1}^{t_i}q_j\phi(w_j^Tx+b_j)\\
+\sum_{i=t_{k^0+1}+1}^{k}a_j\phi(w_i^Tx+b_i).
\end{array}
\]
So, for a fixed $t$, $f_{\left(\Phi^0_t,\psi_t\right)}=f_{0}$ if and only if 
\[
\begin{array}{ccccccc}
\Phi^0_t=&&&&&&\\
(0,&\underbrace{w_1^0,\cdots,w_1^0}&,\cdots,&\underbrace{w_{k^0}^0,\cdots,w_{k^0}^0},&\underbrace{b_1^0,\cdots,b_1^0}&,\cdots,&\underbrace{b_{k^0}^0,\cdots,b_{k^0}^0},\\
&t_2-t_1& &t_{k^0+1}-t_{k^0}&t_2-t_1& &t_{k^0+1}-t_{k^0}\\
&\underbrace{0,\cdots,0}&\underbrace{0,\cdots,0}).&&&&\\
&k^0&k-t_{k^0+1}&&&&
\end{array}
\]
Now, the second derivative of the transfer function is bounded and a constant $C$ exists so that we have the following inequalities:
\[
\forall(\theta_i,\theta_j)\in\left\{b_1,\cdots,b_k,w_{11},\cdots,w_{kd}\right\}^2,\ \sup_{\theta\in\Theta}\Vert \frac{\partial^2 f_\theta(X)}{\partial \theta_i\partial \theta_j}\Vert\leq C(1+\Vert X\Vert^2).
\] 
So,  thanks to assumption H-2, the second order derivative of the function $f_{\left(\Phi_t,\psi_t\right)}$ with respect to the components of $\Phi_t$ will be dominated by a square integrable function.
Then, by the Taylor formula around the identifiable parameter $\Phi^0_t$, we get the following expansion for the numerator of generalized derivative function:    
\begin{lemma}\label{dev}
For a fixed $t$, in the neighborhood  of the identifiable parameter $\Phi^0_t$, we get the following approximation:
\[
\begin{array}{l}
f_{(\Phi_t,\psi_t)}(x)-f_0(x)=(\Phi_t-\Phi^0_t)^Tf^{'}_{(\Phi^0_t,\psi_t)}(x)\\
+0.5(\Phi_t-\Phi^0_t)^Tf^{''}_{(\Phi^0_t,\psi_t)}(x)(\Phi_t-\Phi^0_t)+o(\Vert f_{(\Phi_t,\psi_t)}-f_0\Vert_2^2),
\end{array}
\]
with
\[
\begin{array}{l}
(\Phi_t-\Phi^0_t)^Tf^{'}_{(\Phi^0_t,\psi_t)}(x)=
\gamma+\sum_{i=1}^{k^0}s_i\phi({w^0_i}^Tx+b_i^0)\\
+\sum_{i=1}^{k^0}\sum_{j=t_{i}+1}^{t_{i+1}}q_j\left(w_{j}-w^0_{i}\right)^Txa^0_i\phi^{'}({w^0_i}^Tx+b_i^0)\\
+\sum_{i=1}^{k^0}\sum_{j=t_{i}+1}^{t_{i+1}}q_j\left(b_{j}-b^0_{i}\right)a^0_i\phi^{'}({w^0_i}^Tx+b_i^0)\\
+\sum_{i=t_{k^0+1}+1}^{k}a_i\phi({w_i}^Tx+b_i)
\end{array}
\]
and
\[
\begin{array}{l}
(\Phi_t-\Phi^0_t)^Tf^{''}_{(\Phi^0_t,\psi_t)}(x)(\Phi_t-\Phi^0_t)=\\
\sum_{i=1}^{k^0}\sum_{j=t_{i}+1}^{t_{i+1}}q_j(w_{j}-w^0_{i})^Txx^T(w_{j}-w^0_{i})a^0_i\phi^{''}({w^0_i}^Tx+b_i^0)\\
+\sum_{i=1}^{k^0}\sum_{j=t_{i}+1}^{t_{i+1}}q_j(w_{j}-w^0_{i})^Tx(b_{j}-b_i^0)\phi^{''}({w^0_i}^Tx+b_i^0)\\
+\sum_{i=1}^{k^0}\sum_{j=t_{i}+1}^{t_{i+1}}q_j(b_{j}-b^0_{i})^2\phi^{''}({w^0_i}^Tx+b_i^0)\\ 
 +\sum_{i=1}^{k^0}\sum_{j=t_{i}+1}^{t_{i+1}}q_j(w_{j}-w^0_{i})^Txs_{i}\phi^{'}({w^0_i}^Tx+b_i^0)\\
+\sum_{i=1}^{k^0}\sum_{j=t_{i}+1}^{t_{i+1}}q_j(b_{j}-b^0_{i})s_{i}\phi^{'}({w^0_i}^Tx+b_i^0).
\end{array}
\]
\end{lemma}

This development is obtained by a straightforward calculation of the derivatives of $f_{(\Phi_t,\psi_t)}-f_0$ with respect to the components of $\Phi_t$ up to the second order.

So, the numerator of generalized derivative function can be written like (\ref{lrts}), the proposition \ref{polycov} can be applied to this model and the polynomial bound for the growth of  bracketing number shows the Donsker property of generalized derivative functions. Finally, the lemma \ref{dev} and the next section show that for any  $d$ in the limit-set of derivatives ${\mathcal D}$  a sequence of vector $(\Phi_n,\psi_n)_{n=1,\cdots}$ exists so that $\Vert d-d_{f_{(\Phi_n,\psi_n)}}\Vert_2 \rightarrow 0$  as $\Phi_n$ tends to a $\Phi^0_t$  and since the map
\[
\Phi_t\mapsto P(Y-f_{(\Phi_t,\psi_t)}(X))^2
\] 
admits a second-order Taylor expansion  with strictly positive second derivative $\frac{\partial^2P(Y-f_{(\Phi_t,\psi_t)}(X))^2}{\partial \Phi_t^2}$ at $\Phi_t=\Phi^0_t$ , one can apply theorem \ref{asymptotM} and corollary \ref{corollasymp}.

\paragraph{Asymptotic index set}

The set of limit score functions $\mathcal D$ is defined as the set of functions  $d$ so that one can find a sequence $(\Phi_n,\psi_n)_{n=1,\cdots}$ satisfying $\Vert f_{(\Phi_n,\psi_n)}-f_0\Vert_2\rightarrow 0$ and $\Vert d-d_{f_{(\Phi_n,\psi_n)}}\Vert_2\rightarrow 0$. This limit function depends on the development obtained in lemma \ref{dev}. 

Let us define the two principal behaviors for the sequences $f_{(\Phi_n,\psi_n)}$ which influence the form of functions $d$ :

\begin{itemize}
\item If the second order term is negligible with respect to the first one:
\[
f_{(\Phi_n,\psi_n)}-f_0=(\Phi_n-\Phi^0)^Tf^{\prime}_{(\Phi^0_t,\psi_n)}+o(\Vert f_{(\Phi_n,\psi_n)}-f_0\Vert_2).
\]
\item If the second order term is not negligible with respect to the first one:
\[
\begin{array}{l}
f_{(\Phi_n,\psi_n)}-f_0=(\Phi_n-\Phi^0)^Tf^{\prime}_{(\Phi^0_t,\psi_n)}+\\
0.5(\Phi_n-\Phi^0)^Tf^{\prime\prime}_{(\Phi^0,\psi_n)}(\Phi_n-\Phi^0)+o(\Vert f_{(\Phi_n,\psi_n)}-f_0\Vert_2^2).
\end{array}
\]
\end{itemize}
In the first case,  a set $t=\left(t_1,\cdots,t_{k^0+1}\right)$ exists so that the limit function of  $d_{f_{(\Phi_n,\psi_n)}}$ will be in the set:
\[\begin{array}{l}
{\mathcal D}_1=\left\{\Omega\left(\gamma+\sum_{i=1}^{k^0}\epsilon_i\phi({w^0_i}^TX+b_i^0)+\sum_{i=1}^{k^0}\phi^{'}({w^0_i}^TX+b_i^0)({\zeta}_{i}^TX+\alpha_i)\right.\right.\\
\left.+\sum_{i=t_{k^0+1}+1}^{k}\mu_i\phi({w_i}^TX+b_i)\right),\\
\gamma,\epsilon_1,\cdots,\epsilon_{k^0},\alpha_1,\cdots,\alpha_{k^0}\in{\mathcal R},\mu_{t_{k^0+1}+1},\cdots,\mu_k\in{\mathcal R}^+;\\
\zeta_1,\cdots,\zeta_{k^0}\in{\mathcal R}^{d},\\
\left.(w_{{k^0+1}+1},b_{{k^0+1}+1}),\cdots,(w_{k},b_k)\in\Theta\backslash \left\{(w^0_1,b_1^0),\cdots,(w^0_{k^0},b_{k^0}^0)\right\}\right\}
\end{array}
\]

In the second case, an index $i$ exists so that :
\[
\sum_{j=t_{i}+1}^{t_{i+1}}q_j(\nu_{j}-w^0_{i})=0\mbox{ and }\sum_{j=t_{i}+1}^{t_{i+1}}q_j(\eta_{j}-b^0_{i})=0,
\]
otherwise, the second order term will be negligible compared to the first one.
So
\[
\sum_{j=t_{i}+1}^{t_{i+1}}\sqrt{q_j}\times\sqrt{q_j}(\nu_{j}-w^0_{i})=0\mbox{ and }\sum_{j=t_{i}+1}^{t_{i+1}}\sqrt{q_j}\times\sqrt{q_j}(\eta_{j}-b^0_{i})=0.
\]

Hence, a set $t=\left(t_1,\cdots,t_{k^0+1}\right)$ exists so that the set of functions $d$ will be:
\[\begin{array}{l}
{\mathcal D}_2=\left\{\Omega\left(\gamma+\sum_{i=1}^{k^0}\epsilon_i\phi({w^0_i}^TX+b_i^0)\right.\right.\\
\left.+\sum_{i=1}^{k^0}\phi^{'}({w^0_i}^TX+b_i^0)({\zeta}_{i}^TX+\alpha_i)\right.\\
+\delta(i)\sum_{i=1}^{k^0}\phi^{''}({w^0_i}^TX+b_i^0)\times\\
\left(\left(\sum_{j=t_{i}+1}^{t_{i+1}}{\nu_j}^TXX^T\nu_j+\eta_j{\nu_j}^TX+{\eta_j}^2\right)\right)\\
\left.+\sum_{i=t_{k^0+1}+1}^{k}\mu_i\phi({w_i}^TX+b_i)\right),\\
\gamma,\epsilon_1,\cdots,\epsilon_{k^0},\alpha_1,\cdots,\alpha_{k^0}, \eta_{t_1},\cdots,\eta_{t_{k^0+1}}\in{\mathcal R},\\
\mu_{t_{k^0+1}+1},\cdots,\mu_k\in {\mathcal R}^+;\zeta_1,\cdots,\zeta_{k^0},\nu_{t_1},\cdots,\nu_{t_{k^0+1}}\in{\mathcal R}^{d},\\
\left.(w_{{k^0+1}+1},b_{{k^0+1}+1}),\cdots,(w_{k},b_k)\in\Theta\backslash \left\{(w^0_1,b_1^0),\cdots,(w^0_{k^0},b_{k^0}^0)\right\}\right\}
\end{array}
\]
where $\delta(i)=1$ if a vector ${\bf q}$ exists  so that:\\
$q_j\geq0$, $\sum_{j=t_{i}+1}^{t_{i+1}}q_j=1$, $\sum_{j=t_{i}+1}^{t_{i+1}}\sqrt{q_j}\nu_j^t=0$ and $\sum_{j=t_{i}+1}^{t_{i+1}}\sqrt{q_j}\eta_j=0$, otherwise $\delta(i)=0$. 

Hence, the limit index set functions will belong to $\mathcal D$.

Conversely, let $d$ be an element of $\mathcal D$, since function $d$ is not null, one of its component is not equal to 0. Let us assume that this component is $\gamma$, but the proof would be similar with any other component.  The norm of  $d$ is the constant $1$, so any component of $d$ is determined by the ratio: $\frac{\epsilon_1}{\gamma},\cdots,\frac{1}{\gamma}\nu_{k^0+1}$.

Then, since $\Theta$ contains a neighborhood of the parameters realizing the true regression function $f_{0}$, we can chose 
\[
\theta_n=\left(\beta^n,a^n_1,\cdots,a^n_k,w^n_{1},\cdots,w^n_{k},b^n_1,\cdots,b^n_k\right)\mapsto (\Phi_t^n,\psi_t^n)
\]
so that:
\[
\begin{array}{l}
\forall i\in\{1,\cdots,k^0\}\ :\ \frac{s^n_i}{\beta_n-\beta^0}\stackrel{n\rightarrow \infty}{\longrightarrow}\frac{\epsilon_i}{\gamma},\\
\forall i\in\{1,\cdots,k^0\}\ :\ \sum_{j=t_{i-1}+1}^{t_i}\frac{q^n_j}{\beta_n-\beta^0}\left(w^n_j-w_i^0\right)\stackrel{n\rightarrow \infty}{\longrightarrow}\frac{1}{\gamma}\zeta_i,\\
\forall i\in\{1,\cdots,k^0\}\ :\ \sum_{j=t_{i-1}+1}^{t_i}\frac{q^n_j}{\beta_n-\beta^0}\left(b^n_j-b_i^0\right)\stackrel{n\rightarrow \infty}{\longrightarrow}\frac{1}{\gamma}\alpha_i,\\
\forall j\in\{t_1,\cdots,t_{k^0+1}\}\ :\ \frac{\sqrt{q^n_j}}{\beta_n-\beta^0}\left(w^n_j-w_i^0\right)\stackrel{n\rightarrow \infty}{\longrightarrow}\frac{1}{\gamma}\nu_j,\\
\forall j\in\{t_1,\cdots,t_{k^0+1}\}\ :\ \frac{\sqrt{q^n_j}}{\beta_n-\beta^0}\left(b^n_j-b_i^0\right)\stackrel{n\rightarrow \infty}{\longrightarrow}\frac{1}{\gamma}\eta_j,\\
\forall j\in\{t_{k^0+1}+1,\cdots,k\}\ :\ \frac{\sqrt{q^n_j}}{\beta_n-\beta^0}a^n_j\stackrel{n\rightarrow \infty}{\longrightarrow}\frac{1}{\gamma}\mu_j.
\end{array}
\] 
\end{proof}

% AOS,AOAS: If there are supplements please fill:
%\begin{supplement}[id=suppA]
%  \sname{Supplement A}
%  \stitle{Title}
%  \slink[url]{http://lib.stat.cmu.edu/aoas/???/???}
%  \sdescription{Some text}
%\end{supplement}
%\bibliography{mybiblio}

\begin{thebibliography}{00}

\bibitem[Amari, Park and Ozeki(2006)]{Amari}
Amari, S., Park, H., Ozeki, T. (2006). Singularities affect dynamics of learning in Neuromanifolds
{\it Neural computation}, 18, 1007--1065.

\bibitem[Dacunha-Castelle and Gassiat(1997)]{Dacunha} 
Dacunha-Castelle, D., Gassiat, E. (1997)  Testing in locally conic models.
{\it ESAIM Probability and statistics}, 1, 285--317.


\bibitem[Dacunha-Castelle and Gassiat(1999)]{Dacunha2} 
Dacunha-Castelle, D., Gassiat, E. (1999) Testing the order of a model using locally conic parameterization: Population mixtures and stationary ARMA process. {\it The Annals of Statistics}, 27, 1178-1209.

\bibitem[Doukhan P., Massart P., Rio E.]{Doukhan}
Doukhan P., Massart P., Rio E. (1995) Invariance principles for absolutely regular empirical processes, \emph{Ann. Inst. Henri Poincaré (B) Probabilités et Statistiques}, \textbf{31(2)}, 393-427

\bibitem[Fukumizu(1996)]{Fukumizu1996}
Fukumizu, K. (1996) A regularity condition of the information matrix of a multilayer perceptron network.
{\it Neural networks}, 9:5, 871--879.

\bibitem[Fukumizu(2003)]{Fukumizu2003} 
Fukumizu, K. (2003) Likelihood ratio of unidentifiable models and multilayer neural networks
{\it The Annals of Statistics},31, 833--851.

\bibitem[Oltéanu M., Rynkiewicz, J. (2012)]{Olteanu}
Oltéanu, M., Rynkiewicz, J. (2012) Asymptotic properties of autoregressive
regime-switching models.
{\it ESAIM Probability and statistics}, 16, 25--47.

\bibitem[Gassiat(2002)]{Gassiat}
Gassiat, E. (2002) Likelihood ratio inequalities with applications to various mixtures.
{\it Annales de l'Institut Henri Poincar\'e}, 38, 897--906.


\bibitem[Hagiwara and Fukumizu(2008)]{Hagiwara}
Hagiwara K, Fukumizu, K. (2008) Relation between weight size and degree of over-fitting in neural network regression.
{\it Neural networks}, 21, 48--58.


\bibitem[Liu and Shao(2003)]{Liu}
Liu, X, Shao, Y. (2003) asymptotics for likelihood ratio tests under loss of identifiability
{\it The Annals of Statistics}, 31, 807--832.


\bibitem[van der Vaart(1998)]{Vandervaart}
 van der Vaart, A.W. (1998) {\it Asymptotic statistics} Cambridge: Cambridge university Press

\bibitem[White(1992)]{White}
White, H. (1992) {\it Artificial Neural Networks: Approximation and Learning Theory.}
Oxford: Basil Blackwell.
\bibitem[Yao(2000)]{Yao}
Yao, J. (2000) On least square estimation for stable nonlinear AR processes. 
{\it The Annals of Institut of Mathematical Statistics},52, 316--331.
\end{thebibliography}

\end{document}